\documentclass[letterpaper, 10pt]{article}

\usepackage[dvipsnames]{xcolor}

\usepackage{
  amsmath, amsthm, amssymb, mathtools, dsfont,          % Math typesetting
  graphicx, wrapfig, subcaption, float,                            % Figures and graphics formatting
  listings, color, inconsolata, pythonhighlight,               % Code formatting
  fancyhdr, sectsty, hyperref, enumerate, enumitem, framed }   % Headers/footers, section fonts, links, lists

\usepackage{hyperref, lipsum} 
\usepackage{siunitx}
\usepackage[title]{appendix}
\usepackage{cleveref}
\usepackage{multicol}
\usepackage{multirow}
\usepackage{makecell}
\usepackage{newpxtext, newpxmath, inconsolata}
\usepackage{amsfonts}
\usepackage{pgfplots}
\pgfplotsset{compat=1.12}
\usepackage{tkz-fct}
\usepackage{tikz,forest}
\usetikzlibrary{patterns}
\usetikzlibrary{arrows.meta}
\useforestlibrary{edges}

\tikzset{
decision/.style={diamond, inner sep=1pt},
chance/.style={circle, minimum width=10pt, draw=blue, fill=none, thick, inner sep=0pt},
}

\usepackage{tikz-cd}
\usepackage{enumitem}
\usepackage[title]{appendix}
\usepackage{booktabs}
\usepackage{tabularx}
\usepackage{dynkin-diagrams}

\usepackage{geometry}

\usepackage[bottom]{footmisc}
\usepackage{mathtools}
\allowdisplaybreaks

\usepackage[font={it,footnotesize}]{caption}

\hypersetup{colorlinks=true, linkcolor=RoyalBlue, citecolor=RedOrange, urlcolor=ForestGreen}

\usepackage{titlesec}
\titleformat{\section}{\large\bfseries}{\thesection\;\;\;}{0em}{}
\titleformat{\subsection}{\normalsize\bfseries\selectfont}{\thesubsection\;\;\;}{0em}{}

\setlist[itemize]{wide=0pt, leftmargin=16pt, labelwidth=10pt, align=left}

\usepackage[subfigure]{tocloft}

\cftsetindents{section}{0em}{0em}

\makeatletter
\renewcommand{\cftsecpresnum}{\begin{lrbox}{\@tempboxa}}
\renewcommand{\cftsecaftersnum}{\end{lrbox}}
\makeatother

\graphicspath{{Images/}{../Images/}}

\usepackage[mathscr]{euscript}

\newcommand{\lc}{\mathrm{lc}}
\newcommand{\bC}{\mathbb{C}}

\newcommand{\bF}{\mathbb{F}}

\newcommand{\bN}{\mathbb{N}}

\newcommand{\bQ}{\mathbb{Q}}

\newcommand{\bZ}{\mathbb{Z}}

\theoremstyle{definition}
\newtheorem{theorem}{Theorem}[section]
\newtheorem{proposition}[theorem]{Proposition}
\newtheorem{lemma}[theorem]{Lemma}
\newtheorem{corollary}[theorem]{Corollary}
\newtheorem{definition}[theorem]{Definition}
\newtheorem{example}[theorem]{Example}

\newtheorem*{theorem*}{Theorem}
\newtheorem*{proposition*}{Proposition}
\newtheorem*{lemma*}{Lemma}
\newtheorem*{corollary*}{Corollary}
\newtheorem*{definition*}{Definition}
\newtheorem*{example*}{Example}
\newtheorem*{remark*}{Remark}

\theoremstyle{remark}

\pgfkeys{/Dynkin diagram,
edge length=1.5cm,
fold radius=1cm,
indefinite edge/.style={
draw=black,
fill=white,
thin,
densely dashed}}

\usepackage{fancyhdr}
\newenvironment{red}{\relax\color{red}}{\hspace*{.5ex}\relax}
\newenvironment{blue}{\relax\color{blue}}{\hspace*{.5ex}\relax}
\newcommand{\ber}{\begin{red}}
\newcommand{\er}{\end{red}}
\newcommand{\beb}{\begin{blue}}
\newcommand{\eb}{\end{blue}}

\begin{document}

% --- Main title and subtitle ---

\title{Powerful Fibonacci polynomials over finite fields}

% --- Author and date of last update ---

\author{Graeme Bates, Ryan Jesubalan, Seewoo Lee, Jane Lu, and Hyewon Shim}
% \date{\normalsize\vspace{-1ex} Last updated: \today}
\date{}

% --- Add title and table of contents ---

\maketitle

% --- Abstracts ---

% \tableofcontents\label{sec:contents}
\begin{abstract}
    Bugeaud, Mignotte, and Siksek proved that the only perfect powers in Fibonacci sequence are 0, 1, 8, and 144.
    In this paper, we study the polynomial analogue of the problem. Especially, we give a complete characterization of the Fibonacci polynomials that are perfect powers or powerful over finite fields, where there are infinitely many of them.
    We also give similar characterizations for some of Horadam's generalized Lucas polynomial sequences, which include Fibonacci, Lucas, Chebyshev, and Jacobsthal polynomials.
\end{abstract}

% --- Main content: import lectures as subfiles ---
\section{Introduction}
\label{sec:intro}

Fibonacci numbers $F_n$ and Lucas numbers $L_n$ are among the most widely studied objects in mathematics. They are defined recursively by
\begin{align*}
F_0 = 0, &\quad F_1 = 1,\quad F_{n} = F_{n-1} + F_{n-2} \quad(n \ge 2), \\
L_0 = 2, &\quad L_1 = 1,\quad L_{n} = L_{n-1} + L_{n-2} \quad(n \ge 2).
\end{align*}
Many authors have studied various arithmetic properties of these sequences.
For example, Brillhart--Montgomery--Silverman \cite{brillhart1988tables} studied factorizations of Fibonacci and Lucas numbers for small $n$. Dubner--Keller \cite{dubner1999new} studied Fibonacci and Lucas numbers that are prime, and it remains open whether there are infinitely many such numbers.
Some works are also interested in determining all perfect powers in the Fibonacci and Lucas sequences.
For instance, Cohn \cite{cohn1964square} gave an elementary proof showing that $F_0 = 0$, $F_1 = F_2 = 1$, and $F_{12} = 144$ are the only squares.
Bugeaud, Mignotte, and Siksek \cite{bugeaud2006classical} showed that $0$, $1$, $8$, and $144$ are the only perfect powers appearing in the Fibonacci sequence (see Section \ref{subsec:prelim_diophantine}).

An analogue of this problem can be formulated for \emph{Fibonacci polynomials} $F_n(T) \in \bZ[T]$, the polynomial analogues of Fibonacci numbers, defined by
\begin{equation}
\label{eqn:fibo_poly}
    F_0(T) = 0, \quad F_1(T) = 1, \quad F_{n}(T) = T F_{n-1}(T) + F_{n-2}(T).
\end{equation}
For example, several authors have studied the irreducibility of Fibonacci polynomials. Webb and Parberry showed that $F_n(T)$ is irreducible over $\bQ$ if and only if $n$ is prime \cite{webb1969divisibility}.
Bergum and Hoggatt considered the analogous problem for Lucas polynomials and their generalizations in \cite{bergum1974irreducibility}.
In \cite{kitayama2017irreducibility}, Kitayama and Shiomi studied irreducibility over finite fields, giving necessary and sufficient conditions for these polynomials to be irreducible over $\bF_q$.

In this paper, we study Fibonacci polynomials over finite fields that are perfect powers, analogous to the work of Cohn \cite{cohn1964square} and Bugeaud--Mignotte--Siksek \cite{bugeaud2006classical}.
In particular, we give a complete characterization of perfect powers.
\begin{theorem}
\label{thm:mainFn}
    Let $j > 1$. Let $p$ be an odd prime number coprime to \(2j\), $q$ be a power of $p$, and let $a = \mathrm{ord}_{2j}(p)$ denote the order of $p$ in $\left( \mathbb{Z}/2j\mathbb{Z} \right)^\times$. For $n > 0$, $F_n(T)$ is a perfect $j$-th power in $\mathbb{F}_q[T]$ if and only if $n = p^{ak}$ for some $k \in \mathbb{N}$.
\end{theorem}
Similar result is available for \(p = 2\) as well (Theorem \ref{thm:mainFnp2}).
Moreover, we give a characterization of Fibonacci polynomials that are \emph{powerful}, i.e. multiplicities of each irreducible factor are at least two.\footnote{Although we will mostly focus on nonconstant polynomials, constant polynomials will also be considered powerful, since they have no non-unit irreducible prime factors.}
\begin{theorem}
    \label{thm:mainFnpowerful}
    Let $p > 3$ be a prime and $\bF_q$ be a finite field of characteristic $p$.
    Then \(F_n(T)\) is powerful in \(\bF_q[T]\) if and only if \(p \mid n\).
\end{theorem}
Similar result holds for $p = 2$ and $p = 3$ as well (Corollary \ref{cor:powerfulp2} and \ref{cor:powerfulp3}).

These results contrast to the case of the original Fibonacci sequence.
There are only finitely many Fibonacci numbers that are perfect powers or powerful, where the second assertion is true under the ABC conjecture \cite{ribenboim1999abc,yabuta2007abc}.

In fact, we prove similar characterizations for certain \emph{generalized Lucas polynomial sequences} of Horadam \cite{Horadam1996}, which generalize both Fibonacci polynomials and Lucas polynomials (see Definition \ref{lucassq}).
The proofs of these characterizations are based on the polynomial analogue of Binet's formula and the discriminant formula of Florez, Higuita, and Ramirez \cite{florez2019resultant}, except that certain small characteristic cases need separate considerations.
We also provide explicit factorizations of these polynomials for certain $p$'s, where the accompanying code can be found in \href{https://github.com/seewoo5/sage-function-field}{github.com/seewoo5/sage-function-field}.

\subsection*{Acknowledgements}

This work was completed during the 2025 Berkeley Math REU program and supported by NSF RTG grant DMS-2342225, MPS Scholars, and the Gordon and Betty Moore Foundation.
We thank Tony Feng for organizing the REU.
\section{Preliminaries}
\label{sec:prelim}

\subsection{Diophantine equations with Fibonacci numbers}
\label{subsec:prelim_diophantine}

Diophantine equations involving Fibonacci numbers have been widely studied.
One such equation asks to find all perfect powers in the Fibonacci sequence, i.e. solutions for the equation
\begin{equation}
\label{eqn:Fnpower}    
F_n = y^p
\end{equation}
for prime \(p\).
Cohn \cite{cohn1964square} proved that the only solutions for \(p = 2\) are \(F_0 = 0^2, F_1 = F_2 = 1^2\) and \(F_{12} = 12^2\) by only using  elementary arguments, and a similar method was used to find all squares in the Lucas sequence \cite{cohn1965lucas}.
Subsequent papers studied \eqref{eqn:Fnpower} for larger \(p\), such as \(p = 3\) \cite{london1969fibonacci,petho1983full}, \(p = 5\) \cite{petho1981perfect}, and \(p < 5.1 \cdot 10^{17}\) \cite{petho2001diophantine}.
In \cite{bugeaud2006classical}, Bugeaud, Mignotte, and Siksek finally proved that these are all such numbers, using the modular method \cite{wiles1995modular,taylor1995ring} and Baker's theorem of linear forms \cite{shorey1986exponential}.
\begin{theorem}[{\cite[Theorem 1, Theorem 2]{bugeaud2006classical}}]
    The only perfect powers in the Fibonacci sequence are \(F_0 = 0, F_1 = 1, F_2 = 1, F_6 = 8,\) and \(F_{12} = 144\).
    Also, the only perfect powers in the Lucas sequence are \(L_1 = 1\) and \(L_3 = 4\).
\end{theorem}

More generally, one can ask to characterize Fibonacci numbers that are \emph{powerful}, i.e. the multiplicities of each prime factors are at least two.
Assuming ABC conjecture, it was shown that there are only finitely many powerful numbers in Fibonacci sequence \cite{ribenboim1999abc,yabuta2007abc}.

\subsection{Fibonacci polynomials}
\label{subsec:prelim_fibonaccipoly}

Fibonacci polynomials (over any ring) are polynomial analogues of Fibonacci numbers, which are recursively defined as in \eqref{eqn:fibo_poly}.
The first few Fibonacci polynomials are:
\[
F_0(T) = 0, \quad F_1(T) = 1, \quad F_2(T) = T, \quad F_3(T) = T^2 + 1, \quad F_4(T) = T^3 + 2T, \quad\dots
\]
These polynomials satisfy the following analogue of the \emph{Binet's formula}.
\begin{proposition}[{\cite[Lemma 3.1]{kitayama2017irreducibility}}]
    Let $\alpha(T), \beta(T)$ be the roots of $X^2 - TX - 1$.
    Then
    \begin{equation}
    \label{eqn:FnBinet}
        F_n(T) = \frac{\alpha(T)^n - \beta(T)^n}{\alpha(T) - \beta(T)}.
    \end{equation}
\end{proposition}

Over a characteristic $\ne 2$ ring, we can write
    \[
        \alpha(T) = \frac{T + \sqrt{T^2 + 4}}{2}, \quad \beta(T) = \frac{T - \sqrt{T^2 + 4}}{2}
    \]
    and
    \begin{equation}
    \label{eqn:FnBinet2}
        F_{n}(T) = \frac{\left(\frac{1}{2}T + \frac{1}{2}\sqrt{T^2 + 4}\right)^{n} - \left(\frac{1}{2}T - \frac{1}{2}\sqrt{T^2 + 4}\right)^{n}}{\sqrt{T^2 + 4}} \text{.}
    \end{equation}

It is known that the Fibonacci polynomial factors as \cite[p. 478]{koshy}
\begin{equation}
    \label{eqn:fibo_factor}
    F_n(T) = \prod_{k=1}^{n-1} \left(T - 2 i \cos \left(\frac{k\pi}{n}\right)\right) = \prod_{k=1}^{n-1} (T - i (\zeta_{2n}^{k} + \zeta_{2n}^{-k}))
\end{equation}
for $n \ge 2$, so the zeros of $F_n(T)$ in $\bC$ are given by
\begin{equation}
\label{eqn:fibo_zeros}
2i \cos \left(\frac{k\pi}{n}\right) = i(\zeta_{2n}^{k} + \zeta_{2n}^{-k}), \quad 1 \le k \le n - 1
\end{equation}
which are all algebraic integers.

\subsection{Lucas polynomial sequence} 
We can consider more general families of polynomials by replacing the coefficient polynomials and initial terms in \eqref{eqn:fibo_poly} by others, which are studied by Horadam \cite{Horadam1996}.

\begin{definition}[{\cite[Horadam]{Horadam1996}}]\label{lucassq}
    Let $f(T)$ and $g(T)$ be monomials of degree 0 or 1 (over some ring R).
    Define the following recurrence relation:
    \begin{equation}
       \mathcal{W}_n(T) = f(T) \mathcal{W}_{n - 1} (T) + g(T) \mathcal{W}_{n - 2}(T) \text{ for } n > 1 \text{,} \label{eqn:lucas_sequence}
    \end{equation}
    letting $\mathcal{W}_0(T)$ be a constant in $R$ and $\mathcal{W}_1(T) \in R[T]$ be a monomial of degree 0 or 1.
    
    Let $W_n(T)$ and $w_n(T)$ be polynomial sequences satisfying the recurrence relation of $\mathcal{W}_n(T)$ and the initial conditions:
    \begin{align*}
        W_0(T) = 0, &\quad W_1(T) = 1, \\
        w_0(T) = 2, &\quad w_1(T) = f(T) \text{.}
    \end{align*}
\end{definition}
These families are also called as \emph{Fibonacci-type} and \emph{Lucas-type}, respectively (which are denoted as $\mathcal{F}_n(T)$ and $\mathcal{L}_n(T)$ in \cite{florez2019resultant}). For the remainder of this section, let $f(T)$ and $g(T)$ be monomials of degree $0$ or $1$ in $\bQ[T]$ of choice (for the most of the cases, they will be in $\bZ[T]$) and let the sequences $W_n(T)$ and $w_n(T)$ be determined by $f(T)$ and $g(T)$.

Continuing, \cite{Horadam1996} observes the \emph{Binet forms}:
\begin{align}
    W_n(T) &= \frac{\alpha(T)^n - \beta(T)^n}{\Delta(T)} \label{eqn:Wnbinet} \\
    w_n(T) &= \alpha(T)^n + \beta(T)^n \label{eqn:wnbinet}
\end{align}
where
\begin{align}
    \alpha(T) &:= \frac{f(T) + \sqrt{f(T)^2 + 4 g(T)}}{2} \label{eqn:lucas_alpha} \\
    \beta(T) &:= \frac{f(T) - \sqrt{f(T)^2 + 4 g(T)}}{2} \label{eqn:lucas_beta} \\
    \Delta(T) &:= \alpha(T) - \beta(T) = \sqrt{f(T)^2 + 4 g(T)} \label{eqn:lucas_delta} \text{.}
\end{align}

When we work over characteristic 2, the equations \eqref{eqn:lucas_alpha} and \eqref{eqn:lucas_beta} need to be replaced by $\alpha(T) = f(T)$ and $\beta(T) = 0$.

This recurrence relation generalizes many of the known polynomial sequences.
For example, when $f(T) = 2T$ and $g(T) = -1$ (resp. $f(T) = 1$ and $g(T) = 2T$), the polynomial $W_n(T)$ becomes the Chebyshev polynomial of the second kind $U_{n-1}(T)$ (resp. Jacobsthal polynomial $J_n(T)$ \cite{horadam1997jacobsthal}).
When $f(T) = T$ and $g(T) = 1$, the polynomial $w_n(T)$ becomes the \emph{Lucas polynomial} $L_n(T)$, i.e.
\begin{equation}
    \label{eqn:lucas_poly}
    L_0(T) = 2, \quad L_1(T) = T, \quad L_n(T) = T L_{n-1}(T) + L_{n-2}(T)
\end{equation}
which is the polynomial analogue of the Lucas numbers.
See also \cite[Table 1]{Horadam1996} for more examples.

\section{Proof of Theorem \ref{thm:mainFn}}
\label{sec:main}

In this section, we prove Theorem \ref{thm:mainFn} and \ref{thm:mainFnp2}, completely characterizing Fibonacci polynomials over finite fields that are perfect powers. For the remainder of this paper, let $\bF_q$ denote a finite field of characteristic $p$. Also, we will mainly focus on nonconstant polynomials.

\subsection{Factorization over finite fields}

We first find the factorization of $F_n(T)$.

\begin{lemma}
\label{lem:Fpk}
    Over $\bF_q$, we have
    \begin{equation}
    \label{eqn:Fpkodd}
        F_{p^k}(T) = (T^2 + 4)^{\frac{p^k - 1}{2}}
    \end{equation}
    when \(p\) is odd, and
    \begin{equation}
    \label{eqn:Fpkeven}
        F_{2^k}(T) = T^{2^k - 1}
    \end{equation}
    when \(p = 2\).
\end{lemma}
\begin{proof}
    When $k = 0$, this is trivial. Assume $k \ge 1$.
    When \(p\) is odd, by \eqref{eqn:FnBinet2}
    \[
        F_{p^k}(T) = \frac{\alpha(T)^{p^k} - \beta(T)^{p^k}}{\alpha(T) - \beta(T)} = \frac{(\alpha(T) - \beta(T))^{p^k}}{\alpha(T) - \beta(T)} = (\alpha(T) - \beta(T))^{p^{k} -1} = (T^2 + 4)^{\frac{p^k -1}{2}}.
    \]
    For \(p = 2\), we still have \((\alpha(T) - \beta(T))^2 = (\alpha(T) + \beta(T))^2 = T^2\) and \(\alpha(T) - \beta(T) = T\), so by \eqref{eqn:FnBinet}
    \[
        F_{2^k}(T) = \frac{\alpha(T)^{2^k} - \beta(T)^{2^k}}{\alpha(T) - \beta(T)} = (\alpha(T) - \beta(T))^{2^k - 1} = T^{2^k - 1}.
    \]
    Note that we are working over characteristic $p$, so $(\alpha - \beta)^{p^k} = \alpha^{p^k} - \beta^{p^k}$ holds for any $k \ge 1$.
\end{proof}

\begin{lemma}
\label{lem:Fpkm}
    Write \(n = p^k m\) where \(m\) is coprime to \(p\). Then \(F_n(T)\) factors over \(\bF_q\) as
    \begin{equation}
        F_{p^k m}(T) = F_{p^k}(T) F_m(T)^{p^k}.
    \end{equation}
\end{lemma}
\begin{proof}
    By \eqref{eqn:FnBinet},
    \begin{align*}
        F_n(T) &= \frac{\alpha(T)^{p^k m} - \beta(T)^{p^k m}}{\alpha(T) - \beta(T)} \\
        &= \frac{(\alpha(T)^{m} - \beta(T)^m)^{p^k}}{\alpha(T) - \beta(T)} \\
        &= \frac{(\alpha(T) - \beta(T))^{p^k}}{\alpha(T) - \beta(T)}\left(\frac{\alpha(T)^m - \beta(T)^m}{\alpha(T) - \beta(T)}\right)^{p^k} \\
        &= F_{p^k}(T) F_m(T)^{p^k}.
    \end{align*}
\end{proof}

By Lemma \ref{lem:Fpk} and \ref{lem:Fpkm}, we get:

\begin{proposition}
\label{prop:Fpkm}
    Write \(n = p^k m\) where \(m\) is coprime to \(p\). Then \(F_n(T)\) factors over \(\bF_q\) as
    \begin{equation}
    \label{eqn:Fpkmfactor}
        F_{p^k m}(T) = (T^2 + 4)^{\frac{p^k - 1}{2}} F_m(T)^{p^k}.
    \end{equation}
    When $p = 2$, the above factorization becomes
    \begin{equation}
        \label{eqn:F2kmfactor}
        F_{2^k m}(T) = T^{2^k - 1} F_m(T)^{2^k}.
    \end{equation}
\end{proposition}

\subsection{Square-free factor}

By Proposition \ref{prop:Fpkm}, it is enough to understand the factorization of $F_m(T)$ for $p \nmid m$.
For odd $p$, the following shows that they are square-free. 
% Next, we prove that
\begin{proposition}
\label{prop:Fmsqfree}
    Assume that \(p\) is odd and \(p \nmid m\). Then \(F_m(T)\) is square-free over \(\bF_q\).
\end{proposition}

We provide two different proofs of Proposition \ref{prop:Fmsqfree}.
The first proof is based on the discriminant formula of Florez--Higuita--Ramirez \cite{florez2019resultant}.

\begin{proof}[First proof of Proposition \ref{prop:Fmsqfree}]
    Recall that a polynomial over a field has no multiple root in an algebraic closure if and only if the distriminant is nonzero.
    The discriminant of \(F_m(T)\) is given by \cite[Theorem 4]{florez2019resultant}
    \[
        \mathrm{Disc}(F_m) = (-1)^{\frac{(m-2)(m-1)}{2}} 2^{m-1} m^{m-3}
    \]
    and it is nonzero in \(\bF_q\) when \(2 \ne p \nmid m\), so \(F_m(T)\) is square-free.
\end{proof}

The second proof directly shows that the zeros \eqref{eqn:fibo_zeros} are all distinct modulo $p$ (i.e. distinct in $\overline{\mathbb{F}}_p$).

\begin{lemma} \label{distinct}
Let $p$ be an odd prime and let $n > 1$ and $p \nmid n$. For $k \in \mathbb{Z}$ with $0 < |k| < n$,
$$\zeta_{2n}^k \not \equiv 1 \pmod p,$$
i.e. $\zeta_{2n}^k - 1$ is nonzero in $\overline{\bF}_p$.
\end{lemma}

\begin{proof}
Let:
$$f(T) := \prod_{\substack{-n < k \leq n \\ k \neq 0, \text{ } k \in \mathbb{Z}}} (T - \zeta_{2n}^k) \text{.}$$

When $k = n$, we have $\zeta_{2n}^n = -1 \not \equiv 1 \pmod p$. Suppose $\zeta_{2n}^k \equiv 1 \pmod p$ for some $k$ such that $0 < |k| < n$, then $f(1)\equiv 0\pmod p$.
Notice since:
    \begin{align*}
        T^{2n}-1 &= \prod_{\substack{-n < k \le n \\ k \in \mathbb{Z}}}(T-\zeta_{2n}^k)
        = (T - 1) f(T) \text{,}
    \end{align*}
we have:
\begin{align*}
    f(T) &= \frac{T^{2n} - 1}{T - 1} = \underbrace{T^{2n - 1} + T^{2n - 2} + \dots + T + 1}_{2n \text{ terms}} \text{.}
\end{align*}
From this we get $f(1) = 2n$ and $f(1) \equiv 0 \pmod{p}$, contradicting $p \nmid n$.
\end{proof}

\begin{proof}[Second proof of Proposition \ref{prop:Fmsqfree}]
    If $m = 1$, then $F_1(T) = 1$ is square-free by definition. Assume $m > 1$.
    Suppose $F_m(T)$ is not square-free, so that it has a multiple zero in $\overline{\bF}_p$.
    By \eqref{eqn:fibo_zeros}, there exist distinct $k, l \in \{1, 2, \dots, m - 1\}$ such that $\zeta_{2m}^{k} + \zeta_{2m}^{-k} = \zeta_{2m}^{l} + \zeta_{2m}^{-l}$ in $\overline{\bF}_p$.
    Then $(\zeta_{2m}^{k} - \zeta_{2m}^{l})(\zeta_{2m}^{k + l} - 1) = 0$, hence $\zeta_{2m}^{k} = \zeta_{2m}^{l}$ or $\zeta_{2m}^{k} = \zeta_{2m}^{-l}$.
    By Lemma \ref{distinct}, we have $k \equiv l \pmod{2m}$ or $k \equiv -l \pmod{2m}$, which are both impossible.
\end{proof}

Next, we show that $F_m(T)$ for $p \nmid m$ is always coprime to \(T^2 + 4\), hence $(T^2 + 4)^{\frac{p^k - 1}{2}}$ and $F_m(T)^{p^k}$ in \eqref{eqn:Fpkmfactor} do not share common factors. Knowing that the recurrence relation of the Fibonacci polynomials is \textit{Fibonacci-type}, we can apply the following lemma.

\begin{lemma}[{\cite[Lemma 19]{florez2019resultant}}]
\label{lem:florez_mod}
    For $n \ge 0$, we have
    \begin{equation}
        W_n \pmod{f^2 + 4g} = \begin{cases}
            n (-g)^{(n-1)/2} & 2 \nmid n \\
            (-1)^{(n+2)/2} (nfg^{(n-2)/2}) / 2 & 2 \mid n
        \end{cases}
        ,
    \end{equation}
    where $f$ and $g$ are the coefficient polynomials in \eqref{eqn:lucas_sequence}.
\end{lemma}

\begin{corollary}\label{cor:alpha}
    Let $p$ be an odd prime. We have
    \begin{equation}
        \label{eqn:Fnalpha}
        F_n(T) \pmod{T^2 + 4} = \begin{cases}
            (-1)^{\frac{n-1}{2}}n & 2 \nmid n \\
            (-1)^{\frac{n}{2} + 1} \frac{nT}{2} & 2 \mid n
        \end{cases}
    \end{equation}
    over a field of characteristic $p$.
\end{corollary}
\begin{proof}
    This is a special case of Lemma \ref{lem:florez_mod}, when $f(T) = T$ and $g(T) = 1$.
    Note that 2 is invertible in a field of odd characteristic.
\end{proof}

\begin{corollary}\label{cor:coprime}
     Suppose $p \nmid m$. Then, $T^2 + 4$ and $F_m(T)$ are coprime in $\mathbb{F}_q[T]$.
\end{corollary}
\begin{proof}
    Assume $p$ is an odd prime.
    Suppose $T^2 + 4$ and $F_m(T)$ in $\mathbb{F}_q[T]$ share a factor. Then, there must exist some common zero of the polynomial, which we will denote as $\alpha \in \bF_{q^2}$.
    Then \(\alpha^2 + 4 = 0\) and \(F_m(\alpha) = 0\).
    However, such \(\alpha\) cannot exist by Corollary \ref{cor:alpha}; if \(m\) is even, then \(F_m(\alpha) = (-1)^{\frac{m}{2}+1} \frac{m}{2} \alpha\) and it is nonzero since \(m \ne 0\) in \(\bF_q\), and the case where \(m\) is odd follows similarly.

    For \(p = 2\), \(F_{n}(0) = F_{n-2}(0)\) for all \(n \ge 2\), so \(F_m(0) = F_1(0) = 1\) for all odd \(m\). Hence \(T^2 + 4 = T^2\) and \(F_m(T)\) are coprime for odd \(m\).
\end{proof}

%source
\begin{proposition}
    \label{prop:Fmp2sq}
    Let \(p = 2\). For $n > 0$, \(F_n(T)\) is square in \(\bF_q[T]\) if and only if \(n\) is odd.
\end{proposition}
\begin{proof}
    This immediately follows from \cite[Lemma 3.5]{kitayama2017irreducibility}; we have \(F_{2k}(T) = T F_k(T)^2\) and \(F_{2k+1}(T) = (F_{k+1}(T) + F_k(T))^2\), which can be proven by \eqref{eqn:FnBinet}.
\end{proof}

Finally, before proceeding to the main theorem, we note the following result.

\begin{theorem}
    \label{thm:Fnnonpower}
    Let $p$ be an odd prime number. If $j > 1$ and $p \mid j$, then there exists no $n$ such that $F_n(T)$ is a perfect $j$-th power in $\mathbb{F}_q[T]$.
\end{theorem}

\begin{proof}
    Suppose $j > 1$ and $p \mid j$. Assume there exists some $n = p^km$ with $k \geq 0$ and $p \nmid m$ such that $F_n(T)$ is a perfect $j$-th power in $\mathbb{F}_q[T]$. 
    If $k = 0$, then $p \nmid m$ and $F_n(T) = F_m(T)$ is square-free by Proposition \ref{prop:Fmsqfree}, which cannot be a perfect $j$-th power for $j > 1$, so $k \ge 1$.
    By Lemma \ref{lem:Fpkm}, we have $F_{p^k m}(T) = (T^2 + 4)^{\frac{p^k - 1}{2}} F_m(T)^{p^k}$, where $F_m(T)$ is coprime to $T^2 + 4$ by Corollary \ref{cor:coprime}.
    Then $j$ should divide the exponent $\frac{p^k - 1}{2}$, which contradicts to $p \mid j$.
\end{proof}

\subsection{Proof of the main theorem}

Now we are ready to prove Theorem \ref{thm:mainFn}.

\begin{proof}[Proof of Theorem \ref{thm:mainFn}]
    $(\Leftarrow)$ Let $k \in \mathbb{N}$ be arbitrary. By Lemma \ref{lem:Fpk}, $F_{p^{ak}}(T) = (T^2 +4)^{\frac{p^{ak} - 1}{2}}$, and by our assumption, $p^a \equiv 1 \pmod {2j}$ and thus $\frac{p^{ak} - 1}{2} \equiv 0 \pmod j$. Therefore, $F_n(T)$ where $n = p^{ak}$ is a perfect $j$-th power.
    
    $(\Rightarrow)$ Suppose $F_n(T)$ is a perfect $j$-th power in $\mathbb{F}_q[T]$ and write $n = p^{ak + b}m$, where $p$ and $m$ are coprime (so $p \nmid m$), $k \geq 0$, and \(0 \le b < a\). By Proposition \ref{prop:Fpkm}, we have
    \[
        F_{p^{ak + b}m}(T) = (T^2 + 4)^\frac{p^{ak + b} - 1}{2} \left(F_m(T)\right)^{p^{ak + b}} \text{.}
    \]
    Suppose $b > 0$. Then, $p^{ak + b} = p^{ak}p^b \equiv p^b \pmod{2j}$, where $p^b \not \equiv 1 \pmod{2j}$ by the assumption that $0 < b < a$. This implies that $\frac{p^{ak + b} - 1}{2} \not \equiv 0 \pmod j$, so $(T^2 + 4)^\frac{p^{ak + b} - 1}{2}$ is not a perfect $j$-th power.
    Since \(T^2 + 4\) and \(F_m(T)\) are coprime by Corollary \ref{cor:coprime}, we get a contradiction and we should have \(b = 0\).
    We can rewrite \(F_n(T)\) as
    \[
        F_{p^{ak}m}(T) = (T^2 + 4)^\frac{p^{ak} - 1}{2} \left(F_m(T)\right)^{p^{ak}} \text{,}
    \]
    where \((T^2 + 4)^{\frac{p^{ak} - 1}{2}}\) is a perfect \(j\)-th power and \(F_m(T)\) is square-free by Proposition \ref{prop:Fmsqfree}. If \(m \ne 1\), then \(F_m(T) \ne 1\) and by Corollary \ref{cor:coprime}, the exponent \(p^{ak}\) needs to be a multiple of \(j > 1\), which contradicts to \( \gcd(p, 2j) = 1\).
    This implies \(m = 1\) and we have \(n = p^{ak}\).
\end{proof}

When \(p = 2\), we already showed that \(F_n(T)\) for $n > 0$ is square in \(\bF_q[T]\) if and only if \(n\) is odd. We can further determine higher powers in \(\bF_q[T]\).
\begin{theorem}
\label{thm:mainFnp2}
    Let \(j > 2\) be an odd number and \(a = \mathrm{ord}_j(2)\) be the order of \(2\) in \((\bZ/j\bZ)^\times\).
    Let $q$ be a power of $2$.
    Then \(F_n(T)\) is a perfect \(j\)-th power in $\bF_q[T]$ if and only if \(n = 2^{ak}\) for some \(k \in \bN\).
\end{theorem}
\begin{proof}
    If $n = 2^{ak}$, then $2^a \equiv 1 \pmod{j}$ and \eqref{eqn:Fpkeven} shows that $F_{n}(T) = T^{2^{ak} - 1} = (T^{\frac{2^{ak} - 1}{j}})^{j}$ is a perfect $j$-th power.
    Conversely, assume that $n = 2^{s} \cdot m$ for $s \ge 0$ and odd $m$.
    From \eqref{eqn:fibo_poly}, one can easily check that $F_m(0) = 1$ for all odd $m$.
    In particular, $\gcd(T, F_m(T)) = 1$.
    If $m > 1$, then by \eqref{eqn:F2kmfactor} $F_n(T)$ cannot be a perfect $j$-th power.
    Hence $m = 1$ and $F_n(T) = F_{2^s}(T) = T^{2^{s} - 1}$, which is a perfect $j$-th power if and only if $2^s \equiv 1 \pmod{j} \Leftrightarrow s \equiv 0 \pmod{a}$.
\end{proof}

Similarly, we can also prove Theorem \ref{thm:mainFnpowerful}.
\begin{proof}[Proof of Theorem \ref{thm:mainFnpowerful}]
    $(\Rightarrow)$ If $p \nmid n$, then $F_n(T)$ is square-free and cannot be a powerful polynomial.

    $(\Leftarrow)$ If $n = p^k m$ with $k \ge 1$, then $F_n(T)$ is powerful by \eqref{eqn:Fpkmfactor}.
    Here we use $p > 3$ to guarantee $\frac{p^k - 1}{2} > 1$.
\end{proof}

We need extra care for $p = 3$, since we can have $\frac{p^k - 1}{2} = 1$ when $k = 1$.
\begin{corollary}
    \label{cor:powerfulp3}
    For $p = 3$, $F_n(T)$ is powerful over $\bF_q$ if and only if $9 \mid n$.
\end{corollary}
\begin{proof}
    The proof is similar to that of Theorem \ref{thm:mainFnpowerful}, except that we need $k \ge 2$ to have $\frac{3^k - 1}{2} > 1$.
\end{proof}

Similar characterization is also available for $p = 2$.
\begin{corollary}
\label{cor:powerfulp2}
    For $p = 2$, $F_n(T)$ is powerful over $\bF_q$ if and only if $n \ge 3$ is odd or a multiple of $4$.
\end{corollary}
\begin{proof}
    By \eqref{eqn:F2kmfactor} and \cite[Lemma 3.5]{kitayama2017irreducibility}, when $n = 2^k m$ with odd $m$, we have 
    \[F_n(T) = T^{2^k - 1}F_{m}(T)^{2^k} = T^{2^k - 1}(F_{(m+1)/2}(T)  + F_{(m-1)/2}(T))^{2^{k+1}}.\]
    If $k = 0$ and $n = m \ge 3$ is odd, then $F_n(T) = (F_{(m+1)/2}(T)  + F_{(m-1)/2}(T))^{2}$ is a square of a non-constant polynomial, hence powerful.
    Similarly, $F_n(T)$ is powerful when $k \ge 2$.
    However, if $k = 1$, then $F_n(T) = TF_m(T)^2$ is not powerful, since $T$ and $F_m(T)$ are coprime and $F_n(T)$ is divisible by $T$ only once.
\end{proof}

Table \ref{tab:fibo} shows the factorizations of $F_n(T)$ over $\bF_p$ for $p = 2, 3, 5$ and $1 \le n \le 20$, which is consistent with the above characterizations.
\section{Generalization}

In this section, we generalize the results for the Fibonacci polynomials to $W_n(T)$ \emph{with $\deg f(T) = 1$ and $\deg g(T) = 0$}, including the Chebyshev polynomials of the second kind (Example \ref{ex:chebyshev2}). We also find all the polynomials in the Lucas-type sequence $w_n(T)$ which are perfect powers. 

First, we need the following factorization lemmas, similar to Lemma \ref{lem:Fpk} and \ref{lem:Fpkm}.
We skip the details since the proofs are similar.

\begin{lemma}
\label{lem:Wwpk}
    Let $p$ be a prime and $n = p^k$ for some $k \in \mathbb{N}$. Then $W_n(T) = \Delta(T)^{n-1}$ and $w_n(T) = f(T)^n$ in $\bF_q[T]$.
\end{lemma}

\begin{lemma}
\label{lem:Wwpkm}
    Let $n = p^km$ for some $k \in \mathbb{N}$. Then, $W_n(T) = W_{p^k}(T) W_{m} (T)^{p^k} = \Delta(T)^{p^k - 1} W_{m}(T)^{p^k}$ and $w_n(T) = w_m(T)^{p^k}$ in $\mathbb{F}_q[T]$.
    In particular, $w_n(T)$ is a perfect power if and only if it is powerful.
\end{lemma}

We also have the following coprimality result analogous to Corollary \ref{cor:coprime}.
\begin{proposition}
\label{prop:Wmcoprime}
    Suppose $p$ is an odd prime and $p \nmid m \cdot \lc(f) \cdot g$. Then $f(T)^2 + 4g(T)$ and $W_m(T)$ are coprime in $\mathbb{F}_q[T]$.
\end{proposition}

\begin{proof}
    Suppose $f(T)^2 + 4g(T)$ and $W_m(T)$ in $\mathbb{F}_q[T]$ share a factor. Then, there must exist some common zero of the polynomial, which we will denote as $\alpha \in \bF_{q^2}$.
    Then $f(\alpha)^2 + 4g(\alpha) = 0$ and \(W_m(\alpha) = 0\).
    However, such \(\alpha\) cannot exist by Lemma \ref{lem:florez_mod}.
\end{proof}

We will now let $\deg f = 1$ and $\deg g = 0$ in the recursive definition of $W_n(T)$ and $w_n(T)$ (see Definition \ref{lucassq}). Using the discriminant formulas in \cite{florez2019resultant}, we can show the following results analogous to the case of Fibonacci polynomials.

\begin{proposition}
    \label{prop:Wwsqfree}
    Suppose $p$ is odd and $p \nmid m \cdot \lc(f) \cdot g$, where $\lc(f)$ is the leading coefficient of $f(T)$. Then, $W_m(T)$ and $w_m(T)$ are square-free over $\mathbb{F}_q$.
\end{proposition}

\begin{proof}
    This follows from the discriminant formulas for $W_m$ and $w_m$ \cite[Theorem 4, Theorem 5]{florez2019resultant}:
    \begin{align*}
        \mathrm{Disc}(W_m) &= (-g)^{\frac{(m-1)(m-2)}{2}} 2^{m-1} m^{m-3} \lc(f)^{(m-1)(m-2)}, \\
        \mathrm{Disc}(w_m) &= (-g)^{\frac{m(m-1)}{2}} 2^{m-1} m^{m} \lc(f)^{m(m-1)} \text{.}
    \end{align*}
    In particular, the discriminants are nonzero for $p \nmid m \cdot \lc(f) \cdot g$, and $W_m$ and $w_m$ are thus square-free over $\mathbb{F}_q$.
\end{proof}

\begin{theorem}
    Let $p$ be an odd prime number and assume $p \nmid \lc(f) \cdot g$. If $j > 1$ and $p \mid j$, then there exists no $n$ such that $W_n(T)$ is a perfect $j$-th power in $\mathbb{F}_q[T]$.
\end{theorem}

\begin{proof}
    One can proceed as in the proof of Theorem \ref{thm:Fnnonpower}.
    Suppose $W_{p^km} (T)$ is a perfect $j$-th power for some $k \ge 0$ and $m$ not divisible by $p$.
    If $k = 0$, $W_{p^k m}(T) = W_{m}(T)$ is square-free by Proposition \ref{prop:Wwsqfree}, which cannot be a perfect $j$-th power for $j > 1$.
    Hence $k \ge 1$, and Lemma \ref{lem:Wwpkm} gives $W_{p^k m}(T) = (f(T)^2 + 4g(T))^{\frac{p^k - 1}{2}} W_m(T)^{p^k}$, where the coprimality of $f^2 + 4g$ and $W_m$ (Proposition \ref{prop:Wmcoprime}) implies that the exponent $\frac{p^k - 1}{2}$ has to be a multiple of $j$.
    This contradicts to $p \mid j$.
\end{proof}

\begin{theorem}
    \label{thm:Wn_perfect_power}
    Let $j > 1$. Let $p$ be an odd prime number coprime to $2j$, and let $a = \mathrm{ord}_{2j}(p)$ denote the order of $p$ in $\left( \mathbb{Z}/2j\mathbb{Z} \right)^\times$. Suppose $p \nmid \lc(f) \cdot g$. For $n > 0$, $W_n(T)$ is a perfect $j$-th power in $\mathbb{F}_q[T]$ if and only if $n = p^{ak}$ for some $k \in \mathbb{N}$.
\end{theorem}

\begin{proof}
    $(\Leftarrow)$ Let $k \in \mathbb{N}$ be arbitrary. By Lemma \ref{lem:Wwpk}, $W_{p^{ak}}(T) = \Delta(T)^{p^{ak} - 1} = \left(f(T)^2 + 4 g(T) \right)^{\frac{p^{ak} - 1}{2}}$. Because $p^a \equiv 1 \pmod{2j}$ and thus $\frac{p^{ak} - 1}{2} \equiv 0 \mod j$, $W_{p^{ak}}(T)$ is a perfect $j$-th power.
    
    $(\Rightarrow)$ Suppose $W_n(T)$ is a perfect $j$-th power in $\mathbb{F}_q[T]$ and write $n = p^{ak + b}m$, where $p$ and $m$ are coprime (so $p \nmid m$), $k \geq 0$, and \(0 \le b < a\). With Lemma \ref{lem:Wwpkm}, we may write $W_n(T)$ as follows:
    \begin{align*}
        W_{p^{ak + b}m}(T) &= W_{p^{ak + b}}(T) \left(W_m(T)\right)^{p^{ak + b}} = \left(f(T)^2 + 4 g(T) \right)^{\frac{p^{ak +b} - 1}{2}} \left(W_m(T)\right)^{p^{ak + b}}
    \end{align*}

    Suppose $b > 0$. Then, $p^{ak + b} \not \equiv 1 \pmod{2j}$ by the assumption $b < a$ and thus $\frac{p^{ak + b} - 1}{2} \not \equiv 0 \pmod j$. So $W_{p^{ak + b}}(T)$ is not a perfect $j$-th power. Then, $W_{p^{ak + b}m}(T)$ is not a perfect $j$-th power since $W_{p^{ak + b}}(T)$ and $W_m(T)$ are coprime by Proposition \ref{prop:Wmcoprime}. We thus let $b = 0$ and rewrite $W_n(T)$ as follows: $$W_{p^{ak}m}(T) = W_{p^{ak}}(T) \left(W_m(T)\right)^{p^{ak}} \text{,}$$ where $W_{p^{ak}}(T)$ is a perfect $j$-th power and $W_m(T)$ is square-free by Proposition \ref{prop:Wwsqfree}. Then, if $m \neq 1$, then $W_m(T) \neq 1$ and by Proposition \ref{prop:Wmcoprime} $j > 1$ must divide $p^{ak}$, which is a contradiction to $\gcd(p,2j) = 1$. We thus have $m = 1$ and $n = p^{ak}$.
\end{proof}

\begin{theorem}
    \label{thm:Wn_powerful}
    Suppose $p > 3$ and assume $p \nmid \lc(f) \cdot g$.
    Then \(W_n(T)\) is powerful in \(\bF_q[T]\) if and only if $n$ is a multiple of $p$.
\end{theorem}

\begin{proof}
    $(\Rightarrow)$ If $p \nmid n$, then $W_n(T)$ is square-free and cannot be a powerful polynomial.

    $(\Leftarrow)$ If $n = p^k m$ with $k \ge 1$, then $W_n(T)$ is powerful by Lemma \ref{lem:Wwpkm}.
\end{proof}

When $p = 2$, we have the following results, generalizing Theorem \ref{thm:mainFnp2} and Corollary \ref{cor:powerfulp2}.
\begin{lemma}
    \label{lem:Wnexpansion}
    For each $n$, there exist integers $\{c_{n, k}\}_{0 \le k \le \frac{n-1}{2}}$ such that
    \begin{equation}
        \label{eqn:Wnexpansion}
        W_n = \sum_{k=0}^{\lfloor\frac{n-1}{2}\rfloor} c_{n, k} f^{n-1 - 2k} g^k
    \end{equation}
    where $c_{n, (n-1)/2} = 1$ for odd $n$.
    In particular, $f \mid W_{n}$ for even $n$.
    When $n$ is odd and $g$ is a nonzero constant, then $\gcd(W_{n}, f) = 1$.
\end{lemma}
\begin{proof}
    The equation \eqref{eqn:Wnexpansion} easily follows from induction on $n$, and the following claims immediately follow from it.
\end{proof}

\begin{lemma}
    \label{lem:Wnsquare}
    Let $q$ be a power of 2 and $g \in \bZ$. Over $\bF_q$, we have
    \begin{enumerate}
        \item $W_{2n}(T) = f(T) W_n(T)^2$
        \item $W_{2n+1}(T) = (W_{n+1}(T) + g W_{n}(T))^2$.
    \end{enumerate}
\end{lemma}
\begin{proof}
    The proof for $F_n(T)$ in \cite[Lemma 3.5]{kitayama2017irreducibility} generalizes.
    The first claim follows from the Binet form \eqref{eqn:Wnbinet} and $\Delta = f$. 
    For the second claim, we have $f(T) W_{2n+1}(T) = W_{2n+2}(T) - g W_{2n}(T) = f(T)(W_{n+1}(T) + g W_{n}(T))^2$ (here we used $g^2 \equiv g \equiv -g\pmod{2}$), and the result follows.
\end{proof}

\begin{theorem}
    Let $p = 2$ and $q$ be a power of 2.
    Let $j > 2$ be an odd number and $a = \mathrm{ord}_j(2)$ be the order of 2 in $(\bZ / j\bZ)^\times$.
    Then $W_n(T)$ is a perfect $j$-th power in $\bF_q[T]$ if and only if $n = 2^{ak}$ for some $k \ge 1$.
\end{theorem}
\begin{proof}
    It can be proved similarly as Theorem \ref{thm:mainFnp2}.
    For $n = 2^k m$ with odd $m$, note that $W_n(T) = f(T)^{2^k - 1} W_m(T)^{2^k}$ by Lemma \ref{lem:Wwpkm} and \ref{lem:Wnexpansion}.
\end{proof}

\begin{theorem}
    Let $p = 2$ and $q$ be a power of 2.
    $W_n(T)$ is powerful over $\bF_q$ if and only if $n \ge 3$ is odd or a multiple of $4$.
\end{theorem}
\begin{proof}
    The proof is similar to Corollary \ref{cor:powerfulp2}, where we use Lemma \ref{lem:Wwpkm} and \ref{lem:Wnsquare} instead.
\end{proof}

\begin{theorem}
    Let $p = 3$ and $q$ be a power of 3.
    $W_n(T)$ is powerful over $\bF_q$ if and only if $9 \mid n$.
\end{theorem}
\begin{proof}
    The proof is similar to Corollary \ref{cor:powerfulp3}.
\end{proof}

\begin{example}
\label{ex:chebyshev2}
    When $f(T) = 2T$ and $g(T) = -1$, we have $W_n(T) = U_{n-1}(T)$, the Chebyshev polynomial of the second kind.
    By Theorem \ref{thm:Wn_perfect_power} and \ref{thm:Wn_powerful}, when $p$ is odd, $U_n(T)$ is a perfect power (resp. powerful) in $\bF_q[T]$ if and only if $n = p^{\mathrm{ord}_{2j}(p)k} - 1$ for some $k \in \bN$ (resp. $n \equiv -1 \pmod{p}$ for $p > 3$).
    Table \ref{tab:chebyshev2} shows factorizations of $U_n(T)$ over $\bF_3$ and $\bF_5$ for $1 \le n \le 20$.
    Note that $U_{2m}(T) \equiv 1 \pmod{2}$ and $U_{2m+1}(T) \equiv 0 \pmod{2}$ for all $m \ge 0$.
\end{example}

Similarly, we can characterize when $w_n(T)$ is a perfect power.
\begin{corollary}
\label{cor:wn_perfect_power}
    Let $p$ be an odd prime and assume $p \nmid \lc (f) \cdot g$. Then for $n > 0$, $w_n(T)$ is a perfect power in $\mathbb{F}_q[T]$ if and only if $n$ is a multiple of $p$.
\end{corollary}
\begin{proof}
    If $n = p^k m$ for some positive integer $k$ and $m$, $w_n(T)$ is a perfect power by Lemma \ref{lem:Wwpkm}.
    
    Suppose $n$ is not divisible by $p$. By Proposition \ref{prop:Wwsqfree}, $w_n(T)$ is square-free, and thus $w_n(T)$ is not a perfect power.
\end{proof}

\begin{example}
    When $f(T) = T$ and $g(T) = 1$, $w_n(T)$ recovers the Lucas polynomial $L_n(T)$.
    Thus $L_n(T)$ is a perfect power in $\bF_q[T]$ if and only if $p \mid n$.
    Table \ref{tab:lucas} shows factorizations of $L_n(T)$ over $\bF_2$, $\bF_3$, and $\bF_5$ for $1 \le n \le 20$.
\end{example}

\begin{example}
    When $f(T) = 2T$ and $g(T) = -1$, we have $w_n(T) = 2 T_n(T)$ where $T_n(T)$ is the Chebyshev polynomial of the first kind.
    Thus $T_n(T)$ is a perfect power in $\bF_q[T]$ if and only if $p \mid n$.
    Table \ref{tab:chebyshev1} shows factorizations of $T_n(T)$ over $\bF_3$ and $\bF_5$ for $1 \le n \le 20$.
    Note that Bhargava and Zieve studied factorizations of Dickson polynomials $D_n(T, a)$ over finite fields as well \cite{bhargava1999factoring}, which recovers the Chebyshev polynomial as $D_n(2T, 1) = 2 T_n(T)$.
\end{example}

When $f(T) = 1$ and $g(T) = 2T$, we have $W_n(T) = J_n(T)$, the Jacobsthal polynomial \cite{horadam1997jacobsthal}.
Unfortunately, Theorem \ref{thm:Wn_perfect_power} and Theorem \ref{thm:Wn_powerful} do not apply to these polynomials, since Florez--Higuita--Ramirez's discriminant formula \cite[Theorem 4]{florez2019resultant} does not cover the case of $\deg f = 0$ and $\deg g = 1$.
However, we can still determine when $J_n(T)$ is a perfect power or powerful in $\bF_q[T]$, by using the idea of the second proof of Proposition \ref{prop:Fmsqfree}.
Since $J_n(T) \equiv 1 \pmod{2}$ for all $n$, we will assume $p$ is odd.

From the Binet's formula \eqref{eqn:Wnbinet}, we have
\[
J_n(T) = \frac{\left(\frac{1 + \sqrt{1 + 8T}}{2}\right)^n - \left(\frac{1 - \sqrt{1 + 8T}}{2}\right)^n}{\sqrt{1 + 8T}}, \quad \alpha(T) = \frac{1 + \sqrt{1 + 8T}}{2}, \quad \beta(T) = \frac{1 - \sqrt{1 + 8T}}{2}.
\]
We have $\deg J_n = \lceil \frac{n}{2} \rceil - 1$, and the above formula tells us that the zeros $x$ of $J_n$ satisfies $\alpha(x)^n = \beta(x)^n \Leftrightarrow (\alpha(x) / \beta(x))^n = 1$ with $\alpha(x) \ne \beta(x)$.
Thus $x$ has a form of
\[
\frac{1 + \sqrt{1 + 8x}}{1 - \sqrt{1 + 8x}} = \zeta_n^k \Leftrightarrow x = x_k = - \frac{\zeta_n^k}{2(1 + \zeta_n^k)^2},\quad k = 1, 2, \dots, \left \lceil \frac{n}{2} \right \rceil - 1.
\]
One can check that $2x_k$ are algebraic integers for all $k$.
In particular, we can reduce $x_k$ modulo $p$ and get elements $\overline{x}_k \in \overline{\bF}_p$.
Using Lemma \ref{distinct} and $x / (1 + x)^2 = y / (1 + y)^2 \Leftrightarrow (x - y)(1 - xy) = 0$, we can show that all the zeros $\overline{x}_k$ for $k = 1, 2, \dots, \lceil \frac{n}{2} \rceil - 1$ are distinct in $\overline{\bF}_p$ when $p \nmid n$. Hence,
\begin{proposition}
    \label{thm:Jmsqfree}
    Let $p$ be an odd prime and $p \nmid m$. Then $J_m(T)$ is square-free in $\bF_q[T]$.
\end{proposition}

Now, using Lemma \ref{lem:Wwpk}, Lemma \ref{lem:Wwpkm}, and $J_1 = J_2 = 1$, we can characterize the perfect powers and powerful polynomials among $J_n(T)$ over $\bF_q$.
Since the argument is similar to that of Theorem \ref{thm:Wn_perfect_power} and \ref{thm:Wn_powerful}, we omit the details.
\begin{theorem}
\label{thm:Jn}
    Let $p$ be an odd prime and $q$ be a power of $p$.
    Let $j > 1$ and $a = \mathrm{ord}_{2j}(p)$ be the order of $p$ in $(\bZ / 2j \bZ)^\times$.
    Then for $n > 2$, $J_n(T)$ is a perfect $j$-th power in $\bF_q[T]$ if and only if $n = p^{ak}$ or $n = 2p^{ak}$ for some $k \ge 1$.
    Also, $J_n(T)$ is powerful in $\bF_q[T]$ if and only if $p \mid n$ when $p > 3$, or $9 \mid n$ when $p = 3$.
\end{theorem}
Table \ref{tab:jacobsthal} shows factorizations of $J_n(T)$ over $\bF_3$ and $\bF_5$ for $1 \le n \le 27$.
It would be interesting to generalize Theorem \ref{thm:Jn} to other Fibonacci-type polynomials $W_n(T)$ with $\deg f = 0$ and $\deg g = 1$.

\begin{table}
\centering
\begin{tabular}{c|p{0.30\textwidth}|p{0.30\textwidth}|p{0.30\textwidth}}
\toprule
$n$ & $F_n(T)$ over $\mathbb{F}_2$ & $F_n(T)$ over $\mathbb{F}_3$ & $F_n(T)$ over $\mathbb{F}_5$ \\
\midrule
1 & $1$ & $1$ & $1$ \\
2 & $T$ & $T$ & $T$ \\
3 & $(T + 1)^{2}$ & $(T^{2} + 1)$ & $(T + 2)(T + 3)$ \\
4 & $T^{3}$ & $T(T + 1)(T + 2)$ & $T(T^{2} + 2)$ \\
5 & $(T^{2} + T + 1)^{2}$ & $(T^{2} + T + 2)(T^{2} + 2 T + 2)$ & $(T + 1)^{2}(T + 4)^{2}$ \\
6 & $T(T + 1)^{4}$ & $T^{3}(T^{2} + 1)$ & $T(T + 2)(T + 3)(T^{2} + 3)$ \\
7 & $(T^{3} + T^{2} + 1)^{2}$ & $(T^{6} + 2 T^{4} + 1)$ & $(T^{3} + 2 T^{2} + 2 T + 2)(T^{3} + 3 T^{2} + 2 T + 3)$ \\
8 & $T^{7}$ & $T(T + 1)(T + 2)(T^{4} + T^{2} + 2)$ & $T(T^{2} + 2)(T^{4} + 4 T^{2} + 2)$ \\
9 & $(T + 1)^{2}(T^{3} + T + 1)^{2}$ & $(T^{2} + 1)^{4}$ & $(T + 2)(T + 3)(T^{3} + 3 T + 2)(T^{3} + 3 T + 3)$ \\
10 & $T(T^{2} + T + 1)^{4}$ & $T(T^{2} + T + 2)(T^{2} + 2 T + 2)(T^{4} + 2 T^{2} + 2)$ & $(T + 1)^{2}(T + 4)^{2}T^{5}$ \\
11 & $(T^{5} + T^{4} + T^{2} + T + 1)^{2}$ & $(T^{10} + T^{6} + 2 T^{4} + 1)$ & $(T^{5} + 2 T^{4} + 4 T^{3} + T^{2} + 3 T + 2)(T^{5} + 3 T^{4} + 4 T^{3} + 4 T^{2} + 3 T + 3)$ \\
12 & $T^{3}(T + 1)^{8}$ & $T^{3}(T + 1)^{3}(T + 2)^{3}(T^{2} + 1)$ & $T(T + 2)(T + 3)(T^{2} + 2)(T^{2} + 3)(T^{2} + 2 T + 4)(T^{2} + 3 T + 4)$ \\
13 & $(T^{6} + T^{5} + T^{4} + T + 1)^{2}$ & $(T^{6} + 2 T^{2} + 1)(T^{6} + 2 T^{4} + T^{2} + 1)$ & $(T^{2} + T + 1)(T^{2} + T + 2)(T^{2} + 2 T + 3)(T^{2} + 3 T + 3)(T^{2} + 4 T + 1)(T^{2} + 4 T + 2)$ \\
14 & $T(T^{3} + T^{2} + 1)^{4}$ & $T(T^{6} + T^{4} + 2 T^{2} + 1)(T^{6} + 2 T^{4} + 1)$ & $T(T^{3} + 2 T^{2} + 2 T + 2)(T^{3} + 3 T^{2} + 2 T + 3)(T^{6} + 2 T^{4} + 4 T^{2} + 2)$ \\
15 & $(T + 1)^{2}(T^{2} + T + 1)^{2}(T^{4} + T^{3} + 1)^{2}$ & $(T^{2} + 1)(T^{2} + T + 2)^{3}(T^{2} + 2 T + 2)^{3}$ & $(T + 1)^{2}(T + 4)^{2}(T + 2)^{5}(T + 3)^{5}$ \\
16 & $T^{15}$ & $T(T + 1)(T + 2)(T^{4} + T^{2} + 2)(T^{8} + 2 T^{6} + 2 T^{4} + T^{2} + 2)$ & $T(T^{2} + 2)(T^{4} + 4 T^{2} + 2)(T^{8} + 3 T^{6} + T^{2} + 2)$ \\
17 & $(T^{4} + T + 1)^{2}(T^{4} + T^{3} + T^{2} + T + 1)^{2}$ & $(T^{8} + T^{7} + 2 T^{6} + T^{5} + T^{4} + 2 T^{2} + T + 1)(T^{8} + 2 T^{7} + 2 T^{6} + 2 T^{5} + T^{4} + 2 T^{2} + 2 T + 1)$ & $(T^{8} + 2 T^{7} + 2 T^{6} + 2 T^{5} + 3 T + 1)(T^{8} + 3 T^{7} + 2 T^{6} + 3 T^{5} + 2 T + 1)$ \\
18 & $T(T + 1)^{4}(T^{3} + T + 1)^{4}$ & $T^{9}(T^{2} + 1)^{4}$ & $T(T + 2)(T + 3)(T^{2} + 3)(T^{3} + 3 T + 2)(T^{3} + 3 T + 3)(T^{6} + T^{4} + 4 T^{2} + 3)$ \\
19 & $(T^{9} + T^{8} + T^{6} + T^{5} + T^{4} + T + 1)^{2}$ & $(T^{18} + 2 T^{16} + 2 T^{12} + 2 T^{10} + 1)$ & $(T^{9} + 2 T^{8} + 3 T^{7} + 4 T^{6} + T^{5} + 2)(T^{9} + 3 T^{8} + 3 T^{7} + T^{6} + T^{5} + 3)$ \\
20 & $T^{3}(T^{2} + T + 1)^{8}$ & $T(T + 1)(T + 2)(T^{2} + T + 2)(T^{2} + 2 T + 2)(T^{4} + 2 T^{2} + 2)(T^{4} + T^{3} + 2)(T^{4} + 2 T^{3} + 2)$ & $(T + 1)^{2}(T + 4)^{2}T^{5}(T^{2} + 2)^{5}$ \\
\bottomrule
\end{tabular}
\caption{Factorization of $F_n(T)$ over $\mathbb{F}_2$, $\mathbb{F}_3$, and $\bF_5$}
\label{tab:fibo}
\end{table}

\begin{table}
\centering
\begin{tabular}{c|p{0.45\textwidth}|p{0.45\textwidth}}
\toprule
$n$ & $U_n(T)$ over $\mathbb{F}_3$ & $U_n(T)$ over $\mathbb{F}_5$ \\
\midrule
1 & $2T$ & $2T$ \\
2 & $(T + 1)(T + 2)$ & $4(T + 2)(T + 3)$ \\
3 & $2T(T^{2} + 1)$ & $3T(T^{2} + 2)$ \\
4 & $(T^{2} + T + 2)(T^{2} + 2 T + 2)$ & $(T + 1)^{2}(T + 4)^{2}$ \\
5 & $2(T + 1)(T + 2)T^{3}$ & $2T(T + 2)(T + 3)(T^{2} + 3)$ \\
6 & $(T^{3} + T^{2} + T + 2)(T^{3} + 2 T^{2} + T + 1)$ & $4(T^{3} + 2 T^{2} + 2 T + 2)(T^{3} + 3 T^{2} + 2 T + 3)$ \\
7 & $2T(T^{2} + 1)(T^{4} + 2 T^{2} + 2)$ & $3T(T^{2} + 2)(T^{4} + 4 T^{2} + 2)$ \\
8 & $(T + 1)^{4}(T + 2)^{4}$ & $(T + 2)(T + 3)(T^{3} + 3 T + 2)(T^{3} + 3 T + 3)$ \\
9 & $2T(T^{2} + T + 2)(T^{2} + 2 T + 2)(T^{4} + T^{2} + 2)$ & $2(T + 1)^{2}(T + 4)^{2}T^{5}$ \\
10 & $(T^{5} + T^{4} + 2 T^{3} + 1)(T^{5} + 2 T^{4} + 2 T^{3} + 2)$ & $4(T^{5} + 2 T^{4} + 4 T^{3} + T^{2} + 3 T + 2)(T^{5} + 3 T^{4} + 4 T^{3} + 4 T^{2} + 3 T + 3)$ \\
11 & $2(T + 1)(T + 2)T^{3}(T^{2} + 1)^{3}$ & $3T(T + 2)(T + 3)(T^{2} + 2)(T^{2} + 3)(T^{2} + 2 T + 4)(T^{2} + 3 T + 4)$ \\
12 & $(T^{3} + 2 T + 1)(T^{3} + 2 T + 2)(T^{3} + T^{2} + 2 T + 1)(T^{3} + 2 T^{2} + 2 T + 2)$ & $(T^{2} + T + 1)(T^{2} + T + 2)(T^{2} + 2 T + 3)(T^{2} + 3 T + 3)(T^{2} + 4 T + 1)(T^{2} + 4 T + 2)$ \\
13 & $2T(T^{3} + T^{2} + 2)(T^{3} + T^{2} + T + 2)(T^{3} + 2 T^{2} + 1)(T^{3} + 2 T^{2} + T + 1)$ & $2T(T^{3} + 2 T^{2} + 2 T + 2)(T^{3} + 3 T^{2} + 2 T + 3)(T^{6} + 2 T^{4} + 4 T^{2} + 2)$ \\
14 & $(T + 1)(T + 2)(T^{2} + T + 2)^{3}(T^{2} + 2 T + 2)^{3}$ & $4(T + 1)^{2}(T + 4)^{2}(T + 2)^{5}(T + 3)^{5}$ \\
15 & $2T(T^{2} + 1)(T^{4} + 2 T^{2} + 2)(T^{8} + T^{6} + 2 T^{4} + 2 T^{2} + 2)$ & $3T(T^{2} + 2)(T^{4} + 4 T^{2} + 2)(T^{8} + 3 T^{6} + T^{2} + 2)$ \\
16 & $(T^{8} + T^{7} + 2 T^{6} + T^{3} + 2 T^{2} + 2 T + 1)(T^{8} + 2 T^{7} + 2 T^{6} + 2 T^{3} + 2 T^{2} + T + 1)$ & $(T^{8} + 2 T^{7} + 2 T^{6} + 2 T^{5} + 3 T + 1)(T^{8} + 3 T^{7} + 2 T^{6} + 3 T^{5} + 2 T + 1)$ \\
17 & $2(T + 1)^{4}(T + 2)^{4}T^{9}$ & $2T(T + 2)(T + 3)(T^{2} + 3)(T^{3} + 3 T + 2)(T^{3} + 3 T + 3)(T^{6} + T^{4} + 4 T^{2} + 3)$ \\
18 & $(T^{9} + T^{8} + T^{7} + 2 T^{6} + T^{3} + 2 T^{2} + 2 T + 1)(T^{9} + 2 T^{8} + T^{7} + T^{6} + T^{3} + T^{2} + 2 T + 2)$ & $4(T^{9} + 2 T^{8} + 3 T^{7} + 4 T^{6} + T^{5} + 2)(T^{9} + 3 T^{8} + 3 T^{7} + T^{6} + T^{5} + 3)$ \\
19 & $2T(T^{2} + 1)(T^{2} + T + 2)(T^{2} + 2 T + 2)(T^{4} + T^{2} + 2)(T^{4} + T^{3} + T^{2} + 2 T + 2)(T^{4} + 2 T^{3} + T^{2} + T + 2)$ & $3(T + 1)^{2}(T + 4)^{2}T^{5}(T^{2} + 2)^{5}$ \\
20 & $(T + 1)(T + 2)(T^{3} + T^{2} + T + 2)^{3}(T^{3} + 2 T^{2} + T + 1)^{3}$ & $(T + 2)(T + 3)(T^{3} + T + 1)(T^{3} + T + 4)(T^{3} + 2 T^{2} + 1)(T^{3} + 2 T^{2} + 2 T + 2)(T^{3} + 3 T^{2} + 4)(T^{3} + 3 T^{2} + 2 T + 3)$ \\
\bottomrule
\end{tabular}
\caption{Factorization of $U_n(T)$ over $\mathbb{F}_3$ and $\mathbb{F}_5$}
\label{tab:chebyshev2}
\end{table}

\begin{table}
\centering
\begin{tabular}{c|p{0.30\textwidth}|p{0.30\textwidth}|p{0.30\textwidth}}
\toprule
$n$ & $L_n(T)$ over $\mathbb{F}_2$ & $L_n(T)$ over $\mathbb{F}_3$ & $L_n(T)$ over $\mathbb{F}_5$ \\
\midrule
1 & $T$ & $T$ & $T$ \\
2 & $T^{2}$ & $(T + 1)(T + 2)$ & $(T^{2} + 2)$ \\
3 & $T(T + 1)^{2}$ & $T^{3}$ & $T(T^{2} + 3)$ \\
4 & $T^{4}$ & $(T^{4} + T^{2} + 2)$ & $(T^{4} + 4 T^{2} + 2)$ \\
5 & $T(T^{2} + T + 1)^{2}$ & $T(T^{4} + 2 T^{2} + 2)$ & $T^{5}$ \\
6 & $T^{2}(T + 1)^{4}$ & $(T + 1)^{3}(T + 2)^{3}$ & $(T^{2} + 2)(T^{2} + 2 T + 4)(T^{2} + 3 T + 4)$ \\
7 & $T(T^{3} + T^{2} + 1)^{2}$ & $T(T^{6} + T^{4} + 2 T^{2} + 1)$ & $T(T^{6} + 2 T^{4} + 4 T^{2} + 2)$ \\
8 & $T^{8}$ & $(T^{8} + 2 T^{6} + 2 T^{4} + T^{2} + 2)$ & $(T^{8} + 3 T^{6} + T^{2} + 2)$ \\
9 & $T(T + 1)^{2}(T^{3} + T + 1)^{2}$ & $T^{9}$ & $T(T^{2} + 3)(T^{6} + T^{4} + 4 T^{2} + 3)$ \\
10 & $T^{2}(T^{2} + T + 1)^{4}$ & $(T + 1)(T + 2)(T^{4} + T^{3} + 2)(T^{4} + 2 T^{3} + 2)$ & $(T^{2} + 2)^{5}$ \\
11 & $T(T^{5} + T^{4} + T^{2} + T + 1)^{2}$ & $T(T^{5} + T^{4} + T + 2)(T^{5} + 2 T^{4} + T + 1)$ & $T(T^{5} + 2 T^{4} + T^{2} + 4 T + 3)(T^{5} + 3 T^{4} + 4 T^{2} + 4 T + 2)$ \\
12 & $T^{4}(T + 1)^{8}$ & $(T^{4} + T^{2} + 2)^{3}$ & $(T^{4} + 2)(T^{4} + 3 T^{2} + 3)(T^{4} + 4 T^{2} + 2)$ \\
13 & $T(T^{6} + T^{5} + T^{4} + T + 1)^{2}$ & $T(T^{3} + T^{2} + T + 2)(T^{3} + T^{2} + 2 T + 1)(T^{3} + 2 T^{2} + T + 1)(T^{3} + 2 T^{2} + 2 T + 2)$ & $T(T^{4} + 3)(T^{4} + T^{2} + 2)(T^{4} + 2 T^{2} + 3)$ \\
14 & $T^{2}(T^{3} + T^{2} + 1)^{4}$ & $(T + 1)(T + 2)(T^{3} + 2 T + 1)(T^{3} + 2 T + 2)(T^{3} + T^{2} + 2)(T^{3} + 2 T^{2} + 1)$ & $(T^{2} + 2)(T^{6} + T^{5} + 4 T^{4} + 3 T^{3} + 4 T^{2} + 3 T + 1)(T^{6} + 4 T^{5} + 4 T^{4} + 2 T^{3} + 4 T^{2} + 2 T + 1)$ \\
15 & $T(T + 1)^{2}(T^{2} + T + 1)^{2}(T^{4} + T^{3} + 1)^{2}$ & $T^{3}(T^{4} + 2 T^{2} + 2)^{3}$ & $T^{5}(T^{2} + 3)^{5}$ \\
16 & $T^{16}$ & $(T^{16} + T^{14} + 2 T^{12} + T^{10} + T^{2} + 2)$ & $(T^{16} + T^{14} + 4 T^{12} + 2 T^{10} + 2 T^{6} + T^{4} + 4 T^{2} + 2)$ \\
17 & $T(T^{4} + T + 1)^{2}(T^{4} + T^{3} + T^{2} + T + 1)^{2}$ & $T(T^{16} + 2 T^{14} + 2 T^{12} + T^{10} + 2 T^{8} + 2)$ & $T(T^{16} + 2 T^{14} + 4 T^{12} + 2 T^{10} + 2 T^{6} + 4 T^{4} + 4 T^{2} + 2)$ \\
18 & $T^{2}(T + 1)^{4}(T^{3} + T + 1)^{4}$ & $(T + 1)^{9}(T + 2)^{9}$ & $(T^{2} + 2)(T^{2} + 2 T + 4)(T^{2} + 3 T + 4)(T^{6} + T^{4} + 2 T^{3} + 4 T^{2} + T + 4)(T^{6} + T^{4} + 3 T^{3} + 4 T^{2} + 4 T + 4)$ \\
19 & $T(T^{9} + T^{8} + T^{6} + T^{5} + T^{4} + T + 1)^{2}$ & $T(T^{18} + T^{16} + 2 T^{14} + 2 T^{12} + T^{10} + 2 T^{8} + 1)$ & $T(T^{9} + T^{8} + T^{5} + 2 T^{4} + 2 T^{3} + 2 T^{2} + T + 4)(T^{9} + 4 T^{8} + T^{5} + 3 T^{4} + 2 T^{3} + 3 T^{2} + T + 1)$ \\
20 & $T^{4}(T^{2} + T + 1)^{8}$ & $(T^{4} + T + 2)(T^{4} + 2 T + 2)(T^{4} + T^{2} + 2)(T^{4} + T^{3} + T^{2} + T + 1)(T^{4} + 2 T^{3} + T^{2} + 2 T + 1)$ & $(T^{4} + 4 T^{2} + 2)^{5}$ \\
\bottomrule
\end{tabular}
\caption{Factorization of $L_n(T)$ over $\mathbb{F}_2$,  $\mathbb{F}_3$, and $\mathbb{F}_5$}
\label{tab:lucas}
\end{table}

\begin{table}
\centering
\begin{tabular}{c|p{0.45\textwidth}|p{0.45\textwidth}}
\toprule
$n$ & $T_n(T)$ over $\mathbb{F}_3$ & $T_n(T)$ over $\mathbb{F}_5$ \\
\midrule
1 & $T$ & $T$ \\
2 & $2(T^{2} + 1)$ & $2(T^{2} + 2)$ \\
3 & $T^{3}$ & $4T(T^{2} + 3)$ \\
4 & $2(T^{4} + 2 T^{2} + 2)$ & $3(T^{4} + 4 T^{2} + 2)$ \\
5 & $T(T^{4} + T^{2} + 2)$ & $T^{5}$ \\
6 & $2(T^{2} + 1)^{3}$ & $2(T^{2} + 2)(T^{2} + 2 T + 4)(T^{2} + 3 T + 4)$ \\
7 & $T(T^{3} + T^{2} + 2)(T^{3} + 2 T^{2} + 1)$ & $4T(T^{6} + 2 T^{4} + 4 T^{2} + 2)$ \\
8 & $2(T^{8} + T^{6} + 2 T^{4} + 2 T^{2} + 2)$ & $3(T^{8} + 3 T^{6} + T^{2} + 2)$ \\
9 & $T^{9}$ & $T(T^{2} + 3)(T^{6} + T^{4} + 4 T^{2} + 3)$ \\
10 & $2(T^{2} + 1)(T^{4} + T^{3} + T^{2} + 2 T + 2)(T^{4} + 2 T^{3} + T^{2} + T + 2)$ & $2(T^{2} + 2)^{5}$ \\
11 & $T(T^{10} + T^{8} + 2 T^{6} + T^{4} + T^{2} + 1)$ & $4T(T^{5} + 2 T^{4} + T^{2} + 4 T + 3)(T^{5} + 3 T^{4} + 4 T^{2} + 4 T + 2)$ \\
12 & $2(T^{4} + 2 T^{2} + 2)^{3}$ & $3(T^{4} + 2)(T^{4} + 3 T^{2} + 3)(T^{4} + 4 T^{2} + 2)$ \\
13 & $T(T^{6} + 2 T^{2} + 1)(T^{6} + 2 T^{4} + 1)$ & $T(T^{4} + 3)(T^{4} + T^{2} + 2)(T^{4} + 2 T^{2} + 3)$ \\
14 & $2(T^{2} + 1)(T^{6} + T^{4} + 2 T^{2} + 1)(T^{6} + 2 T^{4} + T^{2} + 1)$ & $2(T^{2} + 2)(T^{6} + T^{5} + 4 T^{4} + 3 T^{3} + 4 T^{2} + 3 T + 1)(T^{6} + 4 T^{5} + 4 T^{4} + 2 T^{3} + 4 T^{2} + 2 T + 1)$ \\
15 & $T^{3}(T^{4} + T^{2} + 2)^{3}$ & $4T^{5}(T^{2} + 3)^{5}$ \\
16 & $2(T^{16} + 2 T^{14} + 2 T^{12} + 2 T^{10} + 2 T^{2} + 2)$ & $3(T^{16} + T^{14} + 4 T^{12} + 2 T^{10} + 2 T^{6} + T^{4} + 4 T^{2} + 2)$ \\
17 & $T(T^{16} + T^{14} + 2 T^{12} + 2 T^{10} + 2 T^{8} + 2)$ & $T(T^{16} + 2 T^{14} + 4 T^{12} + 2 T^{10} + 2 T^{6} + 4 T^{4} + 4 T^{2} + 2)$ \\
18 & $2(T^{2} + 1)^{9}$ & $2(T^{2} + 2)(T^{2} + 2 T + 4)(T^{2} + 3 T + 4)(T^{6} + T^{4} + 2 T^{3} + 4 T^{2} + T + 4)(T^{6} + T^{4} + 3 T^{3} + 4 T^{2} + 4 T + 4)$ \\
19 & $T(T^{9} + T^{8} + T^{6} + 2 T^{5} + T^{4} + 2 T^{3} + T^{2} + 2 T + 2)(T^{9} + 2 T^{8} + 2 T^{6} + 2 T^{5} + 2 T^{4} + 2 T^{3} + 2 T^{2} + 2 T + 1)$ & $4T(T^{9} + T^{8} + T^{5} + 2 T^{4} + 2 T^{3} + 2 T^{2} + T + 4)(T^{9} + 4 T^{8} + T^{5} + 3 T^{4} + 2 T^{3} + 3 T^{2} + T + 1)$ \\
20 & $2(T^{4} + 2 T^{2} + 2)(T^{4} + T^{3} + 2 T + 1)(T^{4} + T^{3} + 2 T^{2} + 2 T + 2)(T^{4} + 2 T^{3} + T + 1)(T^{4} + 2 T^{3} + 2 T^{2} + T + 2)$ & $3(T^{4} + 4 T^{2} + 2)^{5}$ \\
\bottomrule
\end{tabular}
\caption{Factorization of $T_n(T)$ over $\mathbb{F}_3$ and $\mathbb{F}_5$}
\label{tab:chebyshev1}
\end{table}

\begin{table}
\centering
\begin{tabular}{c|p{0.45\textwidth}|p{0.45\textwidth}}
\toprule
$n$ & $J_n(T)$ over $\mathbb{F}_3$ & $J_n(T)$ over $\mathbb{F}_5$ \\
\midrule
1 & $1$ & $1$ \\
2 & $1$ & $1$ \\
3 & $2(T + 2)$ & $2(T + 3)$ \\
4 & $(T + 1)$ & $4(T + 4)$ \\
5 & $(T^{2} + 1)$ & $4(T + 2)^{2}$ \\
6 & $2(T + 2)$ & $2(T + 1)(T + 3)$ \\
7 & $2(T^{3} + 2 T + 2)$ & $3(T^{3} + 3 T^{2} + 2)$ \\
8 & $2(T + 1)(T^{2} + T + 2)$ & $2(T + 4)(T^{2} + T + 2)$ \\
9 & $(T + 2)^{4}$ & $(T + 3)(T^{3} + 2 T^{2} + 4 T + 2)$ \\
10 & $2(T^{2} + 1)(T^{2} + 2 T + 2)$ & $4(T + 2)^{2}$ \\
11 & $2(T^{5} + 2 T^{3} + 2 T^{2} + 2)$ & $2(T^{5} + T^{2} + 4 T + 3)$ \\
12 & $2(T + 2)(T + 1)^{3}$ & $2(T + 1)(T + 3)(T + 4)(T^{2} + 2 T + 4)$ \\
13 & $(T^{3} + T^{2} + 2)(T^{3} + 2 T^{2} + 2 T + 2)$ & $4(T^{2} + T + 1)(T^{2} + 3 T + 4)(T^{2} + 4 T + 1)$ \\
14 & $(T^{3} + 2 T + 2)(T^{3} + T^{2} + T + 2)$ & $3(T^{3} + T^{2} + 4 T + 1)(T^{3} + 3 T^{2} + 2)$ \\
15 & $2(T + 2)(T^{2} + 1)^{3}$ & $3(T + 2)^{2}(T + 3)^{5}$ \\
16 & $(T + 1)(T^{2} + T + 2)(T^{4} + T^{3} + T^{2} + 2 T + 2)$ & $4(T + 4)(T^{2} + T + 2)(T^{4} + 4 T^{3} + 3 T + 3)$ \\
17 & $(T^{8} + 2 T^{3} + T^{2} + 1)$ & $(T^{8} + 3 T^{7} + 4 T^{5} + 3 T^{3} + 4 T^{2} + 1)$ \\
18 & $(T + 2)^{4}$ & $4(T + 1)(T + 3)(T^{3} + 2 T^{2} + 4 T + 2)(T^{3} + 4 T^{2} + 3 T + 4)$ \\
19 & $2(T^{9} + T^{4} + 2 T^{3} + 2 T + 2)$ & $2(T^{9} + 3 T^{6} + 2 T^{5} + 3 T^{4} + 2 T + 3)$ \\
20 & $2(T + 1)(T^{2} + 1)(T^{2} + 2 T + 2)(T^{4} + T^{2} + T + 1)$ & $(T + 2)^{2}(T + 4)^{5}$ \\
21 & $(T + 2)(T^{3} + 2 T + 2)^{3}$ & $4(T + 3)(T^{3} + 2 T^{2} + T + 3)(T^{3} + 2 T^{2} + 2 T + 3)(T^{3} + 3 T^{2} + 2)$ \\
22 & $2(T^{5} + 2 T^{3} + 2 T^{2} + 2)(T^{5} + T^{4} + T^{3} + 2 T^{2} + T + 1)$ & $4(T^{5} + T^{2} + 4 T + 3)(T^{5} + 3 T^{3} + 3 T^{2} + T + 3)$ \\
23 & $2(T^{11} + T^{9} + 2 T^{6} + 2 T^{5} + 2 T^{2} + 2)$ & $3(T^{11} + 3 T^{10} + T^{8} + 4 T^{6} + 2 T^{5} + 4 T^{3} + 4 T + 2)$ \\
24 & $(T + 2)(T + 1)^{3}(T^{2} + T + 2)^{3}$ & $(T + 1)(T + 3)(T + 4)(T^{2} + 2)(T^{2} + T + 2)(T^{2} + 2 T + 4)(T^{2} + 3 T + 3)$ \\
25 & $(T^{2} + 1)(T^{10} + T^{8} + 2 T^{7} + 2 T^{6} + 2 T^{5} + T^{4} + T^{3} + 2 T^{2} + T + 1)$ & $(T + 2)^{12}$ \\
26 & $(T^{3} + 2 T + 1)(T^{3} + T^{2} + 2)(T^{3} + 2 T^{2} + 1)(T^{3} + 2 T^{2} + 2 T + 2)$ & $3(T^{2} + 3)(T^{2} + T + 1)(T^{2} + 2 T + 3)(T^{2} + 3 T + 4)(T^{2} + 4 T + 1)(T^{2} + 4 T + 2)$ \\
27 & $2(T + 2)^{13}$ & $2(T + 3)(T^{3} + 2 T^{2} + 4 T + 2)(T^{9} + 3 T^{8} + 2 T^{6} + 2 T^{5} + T^{4} + 4 T^{3} + 3 T + 3)$ \\
\bottomrule
\end{tabular}
\caption{Factorization of $J_n(T)$ over $\bF_3$ and $\mathbb{F}_5$}
\label{tab:jacobsthal}
\end{table}
% \input{src/appendix}

% --- Bibliography ---

% Start a bibliography with one item.
% Citation example: "\cite{williams}".
\newpage
\bibliographystyle{acm} % We choose the "plain" reference style
\bibliography{refs} % Entries are in the refs.bib file

\medskip

{Department of Mathematics, University of California---Berkeley, Berkeley, CA 94720, USA}

Graeme Bates \href{mailto:graemebates@berkeley.edu}{graemebates@berkeley.edu}

Ryan Jesubalan \href{mailto:rpj@berkeley.edu}{rpj@berkeley.edu}

Seewoo Lee \href{mailto:seewoo5@berkeley.edu}{seewoo5@berkeley.edu}

Jane Lu \href{mailto:zijun@berkeley.edu}{zijun@berkeley.edu}

Hyewon Shim \href{mailto:shmhywn@berkeley.edu}{shmhywn@berkeley.edu}

% --- Document ends here ---

\end{document}